\documentclass[11pt,reqno]{amsart}
\usepackage{graphicx}
\pdfoutput=1
\usepackage{amsfonts,amsmath,amssymb}
\usepackage{hyperref}
\usepackage{paralist}
\usepackage[linesnumbered,ruled,vlined,norelsize]{algorithm2e}

\usepackage{thmtools}
\declaretheoremstyle[bodyfont=\normalfont]{noncursive}
\declaretheorem{theorem}

\numberwithin{equation}{section}

\def\1#1{\overline{#1}}
\def\2#1{\widetilde{#1}}
\def\3#1{\widehat{#1}}
\def\4#1{\mathbb{#1}}
\def\5#1{\frak{#1}}
\def\6#1{{\mathcal{#1}}}

\title{Lie symmetries of a third order PDE system}

\author {Reza Dastranj}
\address{Department of Mathematics, Masaryk University in Brno}
\email{dastranj@math.muni.cz}

\keywords{Lie symmetry, Third order PDEs, CR-geometry}

\begin{document}

\maketitle

\date{\today}

\begin{abstract}
In this paper we show that a third order PDE system that is a general form of a CR-geometry PDE system has at most a ten-dimensional Lie symmetry algebra. We also show that this estimate is precise.
\end{abstract}

\section{Introduction}
In this paper we apply the Lie method to study symmetries of the following third PDE system:
\[
    (\MakeUppercase {s}):=\left\{
                \begin{array}{ll}
                  u_{{2}}=F \left( x,y,u_{{}},u_{{1}} \right)\\
                   u_{{1,1,1}}=G \left( x,y,u_{{}},u_{{1}} \right)
                \end{array}
              \right.
  \]

where the coefficient of $u_1^2$ in the first equation is nonzero. This PDE system is a general form of a CR-geometry PDE system, see \cite{[1]}. 
\section{Main results}
Without loss of generality assume that $F$ is represented by a power
series with respect to $u_1$:
 
  \[
    (\MakeUppercase {s}):=\left\{
                \begin{array}{ll}
                  u_{{2}}={\it F1} \left( x,y,u_{{}} \right) u_{{1}}+{\it F2} \left( x,y
,u_{{}} \right) {u_{{1}}}^{2}+{\it F3} \left( x,y,u_{{}},u_{{1}} \right) {u_{{1}}}^{3}\\
                   u_{{1,1,1}}=G \left( x,y,u_{{}},u_{{1}} \right)
                \end{array}
              \right.
  \]
where $F2\neq 0$.

Any infinitesimal Lie symmetry $V$ of $(S)$ can be written
\begin{equation}{\label{21}}
V=\xi(x,y,u)\partial x+\tau (x,y,u)\partial y+\phi(x,y,u) \partial u
\end{equation}
for some smooth functions $\xi$, $\tau$, $\phi$, see\cite{[2],[3],[4],[5]}.

We fix a point $(x,y,u)$ and we call the vector 
\begin{eqnarray}{\label{22}}
C(x,y,u)&=&\Big(\xi(x,y,u), \tau(x,y,u), \phi(x,y,u), \xi_x(x,y,u), \tau_x(x,y,u), \phi_x(x,y,u),\\ \nonumber 
&& \phi_u(x,y,u), \xi_{x,x}(x,y,u), \tau_{x,x}(x,y,u),\phi_{x,x}(x,y,u)\Big)
\end{eqnarray}
the \textit{initial Tayolr coefficients} of the infinitesimal symmetry $V$ at the point $(x,y,u)$.
\begin{theorem}
An infinitesimal Lie symmetry of $(S)$ is uniquely determined by the ten initial Taylor coefficients
\begin{equation}{\label{23}}
C(x,y,u)=(\xi, \tau, \phi, \xi_x, \tau_x, \phi_x, \phi_u, \xi_{x,x}, \tau_{x,x},\phi_{x,x}).
\end{equation}
\end{theorem}
\begin{proof}
By differentiating of the PDE system, we have the following equations:
\begin{eqnarray}{\label{24}}
u_{12}&=&F3_u u_1^4+F3_{u_1} u_1^3 u_{11}+F2_u u_1^3+F3_x u_1^3+3F3 u_1^2 u_{11}+F1_u u_1^2+F2_x u_1^2+\\ \nonumber 
&& 2F2 u_1 u_{11}+F1_x u_1+F1 u_{11},
\end{eqnarray}
and
\begin{eqnarray}{\label{25}}
u_{112}&=&2F1_xu_{11}+2F2u_{11}^2+F1G+6F3_x u_1^2u_{11}+4F2_xu_1u_{11}+7F3_uu_1^3u_{11}+\\ \nonumber
&& 5F2_uu_1^2u_{11}+3F1_uu_1u_{11}+6F3_{u_1}u_1^2u_{11}^2+6F3u_1u_{11}^2+F3_{u_1}u_1^3G+3F3u_1^2G+\\ \nonumber 
&&2F2u_1G+F3_{u_1u_1}u_1^3u_{11}^2+2F3_{xu_1}u_1^3u_{11}+2F3_{uu_1}u_1^4u_{11}+F3_{xx}u_1^3+F1_{xu}u_1^2+\\ \nonumber 
&&F3_{uu}u_1^5+F2_{uu}u_1^4+F1_{uu}u_1^3+F1_{xx}u_1+F2_{xx}u_1^2+2F3_{xu}u_1^4+2F2_{xu}u_1^3.
\end{eqnarray}
The third prolongation of $V$ is given by:
\begin{eqnarray}{\label{26}}
Pr^{(3)}V&=&\xi\partial x+\tau \partial y+\phi \partial u+\phi^x \partial u_1+\phi^y \partial u_2+\phi^{xx} \partial u_{11}+\phi^{xy} \partial u_{12}+\phi^{yy} \partial u_{22}+\\ \nonumber 
&&\phi^{xxx} \partial u_{111}+\phi^{xxy} \partial u_{112}+\phi^{xyy} \partial u_{122}+\phi^{yyy} \partial u_{222}=\\ \nonumber
&&\xi\partial x+\tau \partial y+\phi \partial u-\Big(-\phi_x-\phi_u u_1+(\xi_x+\xi_u u_1)u_1+(\tau_x + \tau_u u_1)u_2\Big)\partial u_1-\\ \nonumber
&&\Big(-\phi_y-\phi_u u_2+(\xi_y+\xi_u u_2)u_1+(\tau_y + \tau_u u_2)u_2\Big)\partial u_2-\Big(\xi_{uu}u_1^3+\tau_{uu}u_1^2u_2+\\ \nonumber 
&&3\xi_u u_1 u_{11}+2\tau_u u_1 u_{12}+\tau_u u_2 u_{11}+2\xi_{xu}u_1^2+2\tau_{xu}u_1u_2-\phi_{uu}u_1^2-\phi_u u_{11}+\\ \nonumber
&&2\xi_x u_{11}+2\tau_x u_{12}-2\phi_{xu}u_1+\xi_{xx}u_1+\tau_{xx}u_2-\phi_{xx}\Big)\partial u_{11}-\Big(-\phi_{xy}-\phi_{xu}u_2-\\ \nonumber
&&(-\xi_{xy}-\xi_{xu}u_2)u_1-(-\tau_{xy}-\tau_{xu}u_2)u_2-(\phi_{yu}+\phi_{uu}u_2+(-\xi_{yu}-\xi_{uu}u_2)u_1+\\ \nonumber
&&(-\tau_{yu}-\tau_{uu}u_2)u_2)u_1-(-\xi_y -\xi_u u_2)u_{11}-(\phi_u -\xi_u u_1-2\tau_u u_2 -\tau_y)u_{12}+(\xi_x +\\ \nonumber
&&\xi_u u_1)u_{12}+(\tau_x + \tau_u u_1)u_{22}\Big)\partial u_{12}+\phi^{yy} \partial u_{22}-\Big(3\tau_u u_{11} u_{12}+4\xi_u u_1 u_{111} +\\ \nonumber
&&\tau_u u_2 u_{111}+3 \tau_u u_1 u_{112}+6\xi_{uu} u_1 ^2 u_{11} +3\tau_{uu}u_1 ^2 u_{12}-3\phi_{uu} u_1 u_{11}+3 \tau_{xxu} u_1 u_2+\\ \nonumber
&&3 \tau_{xuu}u_1^2 u_2+9\xi_{xu}u_1 u_{11}+3\tau_{xu}u_2 u_{11}+6\tau_{xu}u_1u_{12}+\tau_{uuu}u_{1}^3u_2+3\tau_{uu}u_1u_2u_{11}-\\ \nonumber
&&3\phi_{xxu}u_1+3\xi_{xxu}u_1^2+\xi_{xxx}u_1+\tau_{xxx}u_2+3\xi_{xuu}u_1^3-3\phi_{xuu}u_1^2-3\phi_{xu}u_{11}+\\ \nonumber
&&3\xi_{xx}u_{11}+3\tau_{xx}u_{12}+\xi_{uuu}u_1^4-\phi_{uuu}u_1^3-\phi_{xxx}+3\xi_u u_{11}^2-\phi_u u_{111} +3\xi_x u_{111}+\\ \nonumber
&&3 \tau_x u_{112}\Big)\partial u_{111}+\phi^{xxy}\partial u_{112}+\phi^{xyy} \partial u_{122}+\phi^{yyy} \partial u_{222}.
\end{eqnarray}
Now, by applying infinitesimal criterion of invariance under a one-parameter
Lie group of transformations, we obtain the following equations:
\begin{eqnarray}{\label{27}}
&&\xi(-F1_x u_1-F2_x u_1^2-F3_x u_1^3)+\tau(-F1_y u_1-F2_y u_1^2-F3_y u_1^3)+\phi(-F1_u u_1-\\ \nonumber
&&F2_u u_1^2-F3_u u_1^3)+(\phi_x+\phi_u u_1-(\xi_x+\xi_u u_1)u_1-(\tau_x+\tau_u u_1)u_2)(-F1-2F2u_1-\\ \nonumber 
&&F3_{u_1}u_1^3-3F3u_1^2)+\phi_y +\phi_u u_2-(\xi_y +\xi_u u_2)u_1-(\tau_y +\tau_u u_2)u_2=0,
\end{eqnarray}
and
\begin{eqnarray}{\label{28}}
&&\phi_{xxx}-3\tau_u u_{11}u_{12}-4\xi_u u_{1}u_{111}-\tau_u u_{2}u_{111}-3\tau_u u_{1}u_{112}-3\tau_{xxu} u_{1}u_{2}-3\tau_{xuu} u_{1}^2u_{2}-\\ \nonumber 
&&9\xi_{xu} u_{1}u_{11}-3\tau_{xu} u_{2}u_{11}-6\tau_{xu} u_{1}u_{12}-\tau_{uuu} u_{1}^3u_{2}-6\xi_{uu} u_{1}^2u_{11}-3\tau_{uu} u_{1}^2u_{12}+3\phi_{uu} u_{1}u_{11}-\\ \nonumber 
&&3\tau_{uu} u_{1}u_{2}u_{11}-3\xi_{u}u_{11}^2+\phi_{u}u_{111}-3\xi_{x}u_{111}-3\tau_{x}u_{112}-\xi G_x-\tau G_y-\phi G_u-(\phi_x +\phi_u u_1 - \\ \nonumber 
&&(\xi_x+\xi_u u_1)u_1-(\tau_x+\tau_u u_1)u_2)G_{u_1}+3\phi_{xxu}u_1-3\xi_{xxu}u_1^2-\xi_{xxx}u_1-\tau_{xxx}u_2-3\xi_{xuu}u_1^3+\\ \nonumber 
&&3\phi_{xuu}u_1^2+3\phi_{xu}u_{11}-3\xi_{xx}u_{11}-3\tau_{xx}u_{12}-\xi_{uuu}u_1^4+\phi_{uuu}u_1^3=0.
\end{eqnarray}
By replacing $u_{2}$, $u_{111}$, $u_{12}$ and $u_{112}$  by  ${\it F1} \left( x,y,u_{{}} \right) u_{{1}}+{\it F2} \left( x,y
,u_{{}} \right) {u_{{1}}}^{2}+{\it F3} \left( x,y,u_{{}},u_{{1}} \right) {u_{{1}}}^{3}$, $G(x,y,u,u_1)$, (\ref{24}), and (\ref{25}) respectively, whenever they occur, and equating 
the coefiicients of the various powers of $u_{11}$ to zero, we find the determining equations for the symmetry group of the system $(S)$ to be the following: 
\begin{eqnarray}{\label{29}}
&&-\xi_y u_1+F1\xi_x u_1+2F2\xi_x u_1^2+F3_{u_1}\xi_x u_1^4+3F3\xi_x u_1^3+F2\xi_u u_1^3+F3_{u_1}\xi_u u_1^5+\\ \nonumber 
&&2F3\xi_u u_1^4+3F3^2\tau_x u_1^5+2F2^2\tau_x u_1^3+F1^2\tau_x u_1+2F3^2\tau_u u_1^6+F2^2\tau_u u_1^4-F3\tau_y u_1^3-\\ \nonumber
&&F2\tau_y u_1^2-F1\tau_y u_1- F1_{x}\xi u_1-F2_{x}\xi u_1^2-F3_{x}\xi u_1^3-F1_{y}\tau u_1-F2_{y}\tau u_1^2-\\ \nonumber 
&&F3_{y}\tau u_1^3-F1_{u}\phi u_1-F2_{u}\phi u_1^2-F3_{u}\phi u_1^3-2F2\phi_x u_1-F3_{u_1}\phi_x u_1^3-3F3\phi_x u_1^2-\\ \nonumber 
&&F2\phi_u u_1^2-F3_{u_1}\phi_u u_1^4-2F3\phi_u u_1^3-F1\phi_x+ F1F2\tau_u u_1^3+F2F3_{u_1}\tau_u u_1^6+\\ \nonumber 
&&F1F3_{u_1}\tau_u u_1^5+4F1F3\tau_x u_1^3+5F2F3\tau_x u_1^4+F3F3_{u_1}\tau_x u_1^6+3F1F2\tau_x u_1^2+\\ \nonumber
&&F2F3_{u_1}\tau_x u_1^5+F1F3_{u_1}\tau_x u_1^4+2F1F3\tau_u u_1^4+3F2F3\tau_u u_1^5+F3F3_{u_1}\tau_u u_1^7+\phi_y=0, 
\end{eqnarray}
and
\begin{eqnarray}{\label{210}}
&&\xi_u + F1 \tau_u +2F2 \tau_x +(6 F3 \tau_x  +4F2\tau_u)u_{{1}}+(6F3_{u_{{1}}} \tau_x  +9 F3 \tau_u ) u_{{1}}^{2}+\\ \nonumber
&& (F3_{u_{{1}}u_{{1}}} \tau_x  +7 F3_{u_{{1}}} \tau_u)  u_{{1}}^{3}+ F3_{u_{{1}}u_{{1}}} \tau_u u_{{1}}^{4}=0,
\end{eqnarray}
and
\begin{eqnarray}{\label{211}}
&&3\phi_{xu}-3\xi_{xx}-3F1\tau_{xx}-6F1_x\tau_x-9\xi_{xu}u_1+3\phi_{uu}u_1-6\xi_{uu}u_{1}^2- 6F1\tau_{uu}u_1^2-\\ \nonumber && 6F3_{xu_{1}}\tau_x u_1^3-6F3_{uu_{1}}\tau_x u_1^4-3F3_{u_{1}}\tau_{xx} u_1^3-9F3\tau_{xx} u_1^2-6F2\tau_{xx} u_1-6F3_{xu_1}\tau_{u} u_1^4-\\ \nonumber &&6F3_{uu_1}\tau_{u} u_1^5-15F2\tau_{xu} u_1^2-21F3\tau_{xu} u_1^3-9F1\tau_{xu} u_1-6F3_{u_1}\tau_{xu} u_1^4-12F3\tau_{uu} u_1^4-\\ \nonumber &&3F3_{u_1}\tau_{uu} u_1^5-9F2\tau_{uu} u_1^3-21F3_x\tau_{u} u_1^3-15F2_x\tau_{u} u_1^2-12F1_u\tau_{u} u_1^2-9F1_x\tau_{u} u_1-\\ \nonumber &&24F3_u\tau_{u} u_1^4-18F2_u\tau_{u} u_1^3-18F3_x\tau_{x} u_1^2-12F2_x\tau_{x} u_1-21F3_u\tau_{x} u_1^3-\\ \nonumber &&15F2_u\tau_{x} u_1^2-9F1_u\tau_{x} u_1=0,
\end{eqnarray}
and
\begin{eqnarray}{\label{212}}
&&G\phi_u -3G\xi_x -G_{u_1}\phi_x+\phi_{xxx}-6F3_{xu}\tau_{x} u_1^4-6F2_{xu}\tau_{x} u_1^3-6F1_{xu}\tau_{x} u_1^2-\\ \nonumber
&&3F3_{uu}\tau_{x} u_1^5-3F2_{uu}\tau_{x} u_1^4-3F1_{uu}\tau_{x} u_1^3-3F1_{xx}\tau_{x} u_1-3F2_{xx}\tau_{x} u_1^2-3F3_{xx}\tau_{x} u_1^3-\\ \nonumber 
&&F3\tau_{xxx} u_1^3-F2\tau_{xxx} u_1^2-F1\tau_{xxx} u_1-3F3_u\tau_{xx} u_1^4-3F2_u\tau_{xx} u_1^3-3F3_x\tau_{xx} u_1^3-\\ \nonumber
 &&3F1_u\tau_{xx} u_1^2-3F2_x\tau_{xx} u_1^2-3F1_x\tau_{xx} u_1-3F3\tau_{xxu} u_1^4-3F2\tau_{xxu} u_1^3-3F1\tau_{xxu} u_1^2-\\ \nonumber 
&&3F2\tau_{xuu} u_1^4-3F3\tau_{xuu} u_1^5-3F1\tau_{xuu} u_1^3-6F3_u\tau_{xu} u_1^5-6F2_u\tau_{xu} u_1^4-6F3_x\tau_{xu} u_1^4-\\ \nonumber 
&&6F1_u\tau_{xu} u_1^3-6F2_x\tau_{xu} u_1^3-6F1_x\tau_{xu} u_1^2-F2\tau_{uuu} u_1^5-F3\tau_{uuu} u_1^6-F1\tau_{uuu} u_1^4-\\ \nonumber 
&&3F3_u\tau_{uu} u_1^6-3F2_u\tau_{uu} u_1^5-3F3_x\tau_{uu} u_1^5-3F1_u\tau_{uu} u_1^4-3F2_x\tau_{uu} u_1^4-3F1_x\tau_{uu} u_1^3-\\ \nonumber 
&&6F3_{xu}\tau_{u} u_1^5-6F2_{xu}\tau_{u} u_1^4-6F1_{xu}\tau_{u} u_1^3-3F3_{uu}\tau_{u} u_1^6-3F2_{uu}\tau_{u} u_1^5-3F1_{uu}\tau_{u} u_1^4-\\ \nonumber 
&&3F2_{xx}\tau_{u} u_1^3-3F1_{xx}\tau_{u} u_1^2-3F3_{xx}\tau_{u} u_1^4-G_{u_1}\phi_u u_1+G_{u_1}\xi_x u_1+G_{u_1}\xi_u u_1^2-\\ \nonumber 
&&3F1G\tau_x -4G\xi_u u_1+\phi_{uuu}u_1^3+3\phi_{xxu}u_1-3\xi_{xxu}u_1^2-\xi_{xxx}u_1-3\xi_{xuu}u_1^3+3\phi_{xuu}u_1^2-\\ \nonumber
&&\xi_{uuu}u_1^4+F1G_{u_1}\tau_x u_1+F3G_{u_1}\tau_u u_1^4+F2G_{u_1}\tau_u u_1^3+F1G_{u_1}\tau_u u_1^2-3F3_{u_1}G\tau_x u_1^3-\\ \nonumber
&&9F3G\tau_x u_1^2-6F2G\tau_x u_1-3F3_{u_1}G\tau_u u_1^4+F3G_{u_1}\tau_x u_1^3+F2G_{u_1}\tau_x u_1^2-10F3G\tau_u u_1^3-\\ \nonumber
&&7F2G\tau_u u_1^2-4F1G\tau_u u_1-G_x\xi -G_y\tau -G_u\phi=0.
\end{eqnarray}
Assume that
\begin{equation}\label{213}
F3(x,y,u,u_1)=\Sigma_{i={0}}^\infty F3_i(x,y,u)u_1^{i}.
\end{equation}
By equating the coefficient of $u_{1}$ to zero in (\ref{210}), we find
\begin{equation}\label{214}
   3 {\it F3_0}\left( x,y,u \right)  \tau_x    +2\,  {\it F2}\tau_u=0.
\end{equation}
We may just consider $F2\neq 0$ everywhere. Therefore we get
\begin{equation}\label{215}
\tau_u=r\tau_x,
\end{equation}
where $r$ is a certain analytic function in $(x,y,u)$, which intrinsically depends on the right-hand sides analytic functions $F2, F3$.

Using the coefficient of $u_1^2$ in (\ref{29}) we have
\begin{equation}\label{216}
\tau_y=r\xi_x+r\tau_x+r\phi_x+r\phi_u+r\xi+r\tau+r\phi,
\end{equation}
where $r$s are certain analytic functions in $(x,y,u)$, which intrinsically depend (but in a complex manner) on the right-hand sides analytic functions $F1, F2, F3$.

From the coefficient of $u_1^0$ in (\ref{210}) we obtain
\begin{equation}\label{217}
\xi_u=r\tau_u +r\tau_x.
\end{equation}
Using the coefficient of $u_1$ in eq(\ref{29}) we find 
\begin{equation}\label{218}
\xi_y=r\xi_x+r\tau_x+r\phi_x+r\tau_y+r\xi+r\tau+r\phi.
\end{equation}
From the coefficient of $u_1^0$ in  eq(\ref{29}) we get
\begin{equation}\label{219}
\phi_y=r\phi_x.
\end{equation}
Therefore all of the first order derivatives of $\xi, \tau$ and $\phi$ are determined by $C(x,y,u)$.

Now we want to show that all of the second order derivatives of $\xi, \tau$ and $\phi$ are determined by $C(x,y,u)$.

By differentiating with respect to $x$ and $u$ of the eq(\ref{215}) we respectively have
\begin{equation}\label{220}
\tau_{xu}=r\tau_{xx}+r\tau_x,
\end{equation}
and
\begin{equation}\label{221}
\tau_{uu}=r\tau_{xu}+r\tau_x .
\end{equation}
Using (\ref{217}) we obtain
\begin{equation}\label{222}
\xi_{xu}=r\tau_{xu}+r\tau_{xx}+r\tau_x+r\tau_u,
\end{equation}
and
\begin{equation}\label{223}
\xi_{uu}=r\tau_{uu}+r\tau_{xu}+r\tau_x+r\tau_u.
\end{equation}

From eq(\ref{219}) we find
\begin{equation}\label{224}
\phi_{xy}=r\phi_{xx}+r\phi_{x},
\end{equation}
and
\begin{equation}\label{225}
\phi_{yy}=r\phi_{xy}+r\phi_x.
\end{equation}
In eq(\ref{211}), using the coefficient of $u_1^0$ we get
\begin{equation}\label{226}
\phi_{xu}=\xi_{xx}+r\tau_{xx}+r\tau_x,
\end{equation}
and from the coefficient of $u_1$ we have
\begin{equation}\label{227}
\phi_{uu}=3\xi_{xu}+r\tau_{xu}+r\tau_{xx}+r\tau_x+r\tau_u .
\end{equation}
Using (\ref{216}) we obtain
\begin{equation}\label{228}
\tau_{xy}=r\phi_{xu}+r\phi_{xx}+r\xi_{xx}+r\tau_{xx}+r\xi_x+r\tau_x+r\phi_x+r\phi_u+r\xi+r\tau+r\phi.
\end{equation}
From eq(\ref{219}) we find
\begin{equation}\label{229}
\phi_{yu}=r\phi_{xu}+r\phi_x .
\end{equation}
Using eq(\ref{218}) we get
\begin{equation}\label{230}
\xi_{xy}=r\phi_{xx}+r\xi_{xx}+r\tau_{xy}+r\tau_{xx}+r\xi_x+r\tau_x+r\phi_x+r\tau_y+r\xi+r\tau+r\phi .
\end{equation}
From eq(\ref{216}) we have
\begin{equation}\label{231}
\tau_{yy}=r\phi_{xy}+r\phi_{yu}+r\xi_{xy}+r\tau_{xy}+r\xi_x+r\tau_x+r\phi_x+r\phi_u+r\xi_y+r\tau_y+r\phi_y+r\xi+r\tau+r\phi,
\end{equation}
and
\begin{equation}\label{232}
\tau_{yu}=r\phi_{xu}+r\phi_{uu}+r\xi_{xu}+r\tau_{xu}+r\xi_u+r\tau_u+r\xi_x+r\tau_x+r\phi_x+r\phi_u+r\xi+r\tau+r\phi.
\end{equation}
Using eq(\ref{218}) we obtain
\begin{equation}\label{233}
\xi_{yy}=r\tau_{xy}+r\phi_{xy}+r\xi_{xy}+r\tau_{yy}+r\xi_y+r\phi_y+r\xi_x+r\tau_x+r\phi_x+r\tau_y+r\xi+r\tau+r\phi,
\end{equation}
and
\begin{equation}\label{234}
\xi_{yu}=r\tau_{xu}+r\phi_{xu}+r\xi_{xu}+r\tau_{yu}+r\xi_u+r\tau_u+r\phi_u+r\xi_x+r\tau_x+r\phi_x+r\tau_y+r\xi+r\tau+r\phi.
\end{equation}

Now we want to show that all of the Third order derivatives of $\xi, \tau$ and $\phi$ are determined by $C(x,y,u)$.

Assume that
\begin{equation}\label{235}
G(x,y,u,u_1)=\Sigma_{i={0}}^\infty G_i(x,y,u)u_1^{i}.
\end{equation}
From the coefficient of $u_1^0$ in eq(\ref{212})we find
\begin{equation}\label{236}
\phi_{xxx}=r\xi+r\tau+r\phi+r\phi_u+r\xi_x+r\phi_x+r\tau_x.
\end{equation}
From eq(\ref{224}) we get
\begin{equation}\label{237}
\phi_{xxy}=r\phi_{xxx}+r\phi_{xx}+r\phi_x,
\end{equation}
and
\begin{equation}\label{238}
\phi_{xyy}=r\phi_{xxy}+r\phi_{xx}+r\phi_{xy}+r\phi_x.
\end{equation}
Using eq(\ref{219}) we have
\begin{equation}\label{239}
\phi_{yyy}=r\phi_{xyy}+r\phi_{xy}+r\phi_{x}.
\end{equation}
From the coefficient of $u_1^0$ in (\ref{211}) we obtain
\begin{equation}\label{240}
\xi_{xx}-\phi_{xu}+F1\tau_{xx}+r\tau_x=0.
\end{equation}
 So, by differentiating of (\ref{240}) with respect to $u$ we find
\begin{equation}\label{241}
\xi_{xxu}-\phi_{xuu}+F1\tau_{xxu}+r\tau_{xu}+r\tau_{xx}+r\tau_x=0.
\end{equation}
Using the coefficient of $u_1^2$ in (\ref{212}) we get
\begin{equation}\label{242}
\xi_{xxu}-\phi_{xuu}+F1\tau_{xxu}+\dfrac{F2}{3}\tau_{xxx}+r\tau_{xu}+r\tau_{xx}+r\xi_x +r\xi_u +r\tau_x +r\tau_u+r\phi_u+r\phi_x+r\xi+r\tau+r\phi=0 .
\end{equation}
Therefore, from (\ref{241}) and (\ref{242}) we have
\begin{equation}\label{243}
\tau_{xxx}=r\tau_{xu}+r\tau_{xx}+r\xi_x +r\xi_u +r\tau_x +r\tau_u+r\phi_u+r\phi_x+r\xi+r\tau+r\phi.
\end{equation}
Using (\ref{215})we obtain
\begin{equation}\label{244}
\tau_{xxu}=r\tau_{xxx}+r\tau_{xx}+r\tau_{x},
\end{equation}
and
\begin{equation}\label{245}
\tau_{xuu}=r\tau_{xxu}+r\tau_{xu}+r\tau_{xx}+r\tau_{x},
\end{equation}
and
\begin{equation}\label{246}
\tau_{uuu}=r\tau_{xuu}+r\tau_{xu}+r\tau_{x}.
\end{equation}
From (\ref{217}) we find
\begin{equation}\label{247}
\xi_{xxu}=r\tau_{xxu}+r\tau_{xxx}+r\tau_{xu}+r\tau_{xx}+r\tau_{u}+r\tau_{x}.
\end{equation}
Using (\ref{226})
\begin{equation}\label{248}
\phi_{xuu}=\xi_{xxu}+r\tau_{xxu}+r\tau_{xx}+r\tau_{xu}+r\tau_{x}.
\end{equation}
From (\ref{219}) we get
\begin{equation}\label{249}
\phi_{yuu}=r\phi_{xuu}+r\phi_{xu}+r\phi_{x}.
\end{equation}
Using (\ref{217}) we have
\begin{equation}\label{250}
\xi_{xuu}=r\tau_{xuu}+r\tau_{xxu}+r\tau_{xx}+r\tau_{xu}+r\tau_{uu}+r\tau_{x}+r\tau_{u},
\end{equation}
and
\begin{equation}\label{250}
\xi_{uuu}=r\tau_{uuu}+r\tau_{xuu}+r\tau_{xu}+r\tau_{uu}+r\tau_{x}+r\tau_{u}.
\end{equation}
From (\ref{227}) we obtain
\begin{equation}\label{252}
\phi_{uuu}=3\xi_{xuu}+r\tau_{xuu}+r\tau_{xxu}+r\tau_{uu}+r\tau_{xu}+r\tau_{xx}+r\tau_x+r\tau_u.
\end{equation}
Using the coefficient of $u_1$ in (\ref{212}) we find
\begin{equation}\label{253}
3\phi_{xxu}-\xi_{xxx}+r\tau_{xxx}+r\tau_{xx}+r\xi_x+r\xi_u+r\tau_x+r\tau_u+r\phi_x+r\phi_u+r\xi+r\tau+r\phi=0,
\end{equation}
and from (\ref{226}) we get
\begin{equation}\label{254}
\phi_{xxu}=\xi_{xxx}+r\tau_{xxx}+r\tau_{xx}+r\tau_{x}.
\end{equation}
Therefore, using (\ref{253}) and (\ref{254}) we have
\begin{equation}\label{255}
\phi_{xxu}=r\tau_{xxx}+r\tau_{xx}+r\xi_x+r\xi_u+r\tau_x+r\tau_u+r\phi_x+r\phi_u+r\xi+r\tau+r\phi.
\end{equation}
From (\ref{226}) we obtain
\begin{equation}\label{256}
\xi_{xxx}=\phi_{xxu}+r\tau_{xxx}+r\tau_{xx}+r\tau_{x}.
\end{equation}
Using (\ref{216}) we find
\begin{eqnarray}\label{257}
\tau_{xxy}&=&r\phi_{xxx}+r\phi_{xxu}+r\xi_{xxx}+r\tau_{xxx}+r\xi_{xx}+r\tau_{xx}+r\phi_{xx}+r\phi_{xu}+\\ \nonumber 
&&r\xi_x+r\tau_x+r\phi_x+r\phi_u+r\xi+r\tau+r\phi.
\end{eqnarray}
From (\ref{219})we get
\begin{equation}\label{258}
\phi_{xyu}=r\phi_{xxu}+r\phi_{xx}+r\phi_{xu}+r\phi_x.
\end{equation}
Using (\ref{230})we have
\begin{eqnarray}\label{259}
\xi_{xxy}&=&r\phi_{xxx}+r\xi_{xxx}+r\tau_{xxy}+r\tau_{xxx}+r\phi_{xx}+r\xi_{xx}+r\tau_{xy}+r\tau_{xx}+r\xi_x+\\ \nonumber 
&&r\tau_x+r\phi_x+r\tau_y+r\xi+r\tau+r\phi.
\end{eqnarray}
From (\ref{231})we obtain
\begin{eqnarray}\label{260}
\tau_{xyy}&=&r\phi_{xxy}+r\phi_{xyu}+r\xi_{xxy}+r\tau_{xxy}+r\xi_{xx}+r\tau_{xx}+r\phi_{xx}+r\phi_{xu}+r\phi_{xy}+\\ \nonumber 
&&r\phi_{yu}+r\xi_{xy}+r\tau_{xy}+r\xi_x+r\tau_x+r\phi_x+r\phi_u+r\xi_y+r\tau_y+r\phi_y+r\xi+r\tau+r\phi.
\end{eqnarray}
Using (\ref{228}) we find
\begin{eqnarray}\label{261}
\tau_{xyu}&=&r\phi_{xxu}+r\phi_{xuu}+r\xi_{xxu}+r\tau_{xxu}+r\phi_{xu}+r\phi_{xx}+r\xi_{xx}+r\tau_{xx}+r\xi_{xu}+\\ \nonumber 
&&r\tau_{xu}+r\phi_{uu}+r\xi_{u}+r\tau_u+r\xi_x+r\tau_x+r\phi_x+r\phi_u+r\xi+r\tau+r\phi.
\end{eqnarray}
Using (\ref{229}) we get
\begin{equation}\label{262}
\phi_{yyu}=r\phi_{xyu}+r\phi_{xu}+r\phi_{xy}+r\phi_x.
\end{equation}
From (\ref{230}) we have
\begin{eqnarray}\label{263}
\xi_{xyy}&=&r\phi_{xxy}+r\xi_{xxy}+r\tau_{xyy}+r\tau_{xxy}+r\phi_{xx}+r\xi_{xx}+r\tau_{xy}+r\tau_{xx}+r\tau_{yy}+\\ \nonumber 
&&r\phi_{xy}+r\tau_{xy}+r\xi_{xy}+r\phi_{y}+r\xi_{y}+r\xi_x+r\tau_x+r\phi_x+r\tau_y+r\xi+r\tau+r\phi.
\end{eqnarray}
Using (\ref{231}) we obtain
\begin{eqnarray}\label{264}
\tau_{yyy}&=&r\phi_{xyy}+r\phi_{yyu}+r\xi_{xyy}+r\tau_{xyy}+r\phi_{xy}+r\phi_{yu}+r\xi_{xy}+r\tau_{xy}+r\xi_{yy}+\\ \nonumber 
&&r\tau_{yy}+r\phi_{yy}+r\xi_y+r\tau_y+r\phi_y+r\xi_x+r\tau_x+r\phi_x+r\phi_u+r\xi+r\tau+r\phi.
\end{eqnarray}
From (\ref{230}) we find
\begin{eqnarray}\label{265}
\xi_{xyu}&=&r\phi_{xxu}+r\xi_{xxu}+r\tau_{xyu}+r\tau_{xxu}+r\xi_{xu}+r\tau_{xu}+r\phi_{xu}+r\tau_{yu}+r\phi_{xx}+\\ \nonumber 
&&r\xi_{xx}+r\tau_{xy}+r\tau_{xx}+r\xi_x+r\tau_u+r\phi_u+r\xi_u+r\tau_x+r\phi_x+r\tau_y+r\xi+r\tau+r\phi.
\end{eqnarray}
Using (\ref{231}) we get
\begin{eqnarray}\label{266}
\tau_{yyu}&=&r\phi_{xyu}+r\phi_{yuu}+r\xi_{xyu}+r\tau_{xyu}+r\phi_{xy}+r\phi_{yu}+r\xi_{xy}+r\tau_{xy}+\xi_{xu}+\\ \nonumber 
&&r\tau_{xu}+r\phi_{xu}+r\phi_{uu}+r\xi_{yu}+r\tau_{yu}+r\xi_u+r\tau_u+r\xi_x+r\tau_x+r\phi_x+r\phi_u+\\ \nonumber 
&&r\xi_y+r\tau_y+r\phi_y+r\xi+r\tau+r\phi.
\end{eqnarray}
From (\ref{232}) we have
\begin{eqnarray}\label{267}
\tau_{yuu}&=&r\phi_{xuu}+r\phi_{uuu}+r\xi_{xuu}+r\tau_{xuu}+r\phi_{xu}+r\phi_{uu}+r\xi_{xu}+r\tau_{xu}+\\ \nonumber 
&&r\xi_{uu}+r\tau_{uu}+r\xi_u+r\tau_u+r\xi_x+r\tau_x+r\phi_x+r\phi_u+r\xi+r\tau+r\phi.
\end{eqnarray}
Using (\ref{233}) we obtain
\begin{eqnarray}\label{268}
\xi_{yyy}&=&r\tau_{xyy}+r\phi_{xyy}+r\xi_{xyy}+r\tau_{yyy}+r\tau_{xy}+r\phi_{xy}+r\xi_{xy}+r\xi_{yy}+r\tau_{yy}+\\ \nonumber 
&&r\phi_{yy}+r\xi_y+r\phi_y+r\xi_x+r\tau_x+r\phi_x+r\tau_y+r\xi+r\tau+r\phi,
\end{eqnarray}
and
\begin{eqnarray}\label{269}
\xi_{yyu}&=&r\tau_{xyu}+r\phi_{xyu}+r\xi_{xyu}+r\tau_{yyu}+r\tau_{xy}+r\phi_{xy}+r\xi_{xy}+r\tau_{yy}+r\xi_{yu}+\\ \nonumber 
&&r\phi_{yu}+r\xi_{xu}+r\tau_{xu}+r\phi_{xu}+r\tau_{yu}+r\xi_u+r\tau_u+r\phi_u+
r\xi_y+\\ \nonumber 
&&r\phi_y+r\xi_x+r\tau_x+r\phi_x+r\tau_y+r\xi+r\tau+r\phi.
\end{eqnarray}
From (\ref{234}) we find
\begin{eqnarray}\label{270}
\xi_{yuu}&=&r\tau_{xuu}+r\phi_{xuu}+r\xi_{xuu}+r\tau_{yuu}+r\tau_{xu}+r\phi_{xu}+r\xi_{xu}+r\tau_{yu}+r\xi_u+\\ \nonumber 
&&r\tau_u+r\phi_u+r\xi_x+r\tau_x+r\phi_x+r\tau_y+r\xi+r\tau+r\phi.
\end{eqnarray}
From the chain rule, it follows that all of the partial derivatives of
$\xi$, $\tau$ and $\phi$ are determined by $C(x,y,u)$, thus completing the proof.
\end{proof}              

Now we want to show that this estimate is precise. For this, we consider the following PDE system 

\[
    (\MakeUppercase {S_m}):=\left\{
                \begin{array}{ll}
                  u_{{2}}=u_{{1}}^2\\
                   u_{{1,1,1}}=0
                \end{array}
              \right.
  \]
The determining equations for this system are:
\begin{eqnarray}\label{271}
&&\tau_u=0 , \quad \phi_x+\xi_y=0 , \quad 3\phi_{ux}-3\xi_{xx}=0, \quad \xi_u+2\tau_x=0 ,\quad 3\phi_{uxx}-\xi_{xxx}=0, \\ \nonumber 
&&\qquad 5\tau_{ux}+2\xi_{uu}=0, \quad 3\tau_{uux}+\xi_{uuu}=0, \quad 2\xi_x -\tau_y-\phi_u=0 , \quad 3\xi_{xu}+2\tau_{xx}-\phi_{uu}=0,  \\ \nonumber 
&&3\xi_{xuu}+3\tau_{uxx}-\phi_{uuu}=0,  \quad 3\xi_{xxu}+\tau_{xxx}-3\phi_{xuu}=0 , \quad \phi_y=0 , \quad \phi_{xxx}=0 , \quad \tau_u=0.
\end{eqnarray}
By solving the determining equations we have:
\begin{eqnarray}\label{272}
\xi&=&\dfrac{C_1}{2} x^2+\dfrac{C_2y-4C_3u+C_4+C_5}{2}x-(2C_1y+2C_6)u-2C_7y+C_8, \\ \nonumber 
\tau&=&\dfrac{C_2}{2}y^2+(C_1x+C_4)y+\dfrac{C_3}{2}x^2+C_6x+C_9,\\ \nonumber 
\phi&=&-\dfrac{C_2}{8}x^2+(C_1u+C_7)x-2C_3u^2+C_5u+C_{10},
\end{eqnarray}
where $C_1$, $C_2$, ..., $C_{10}$ are arbitrary constants.

So, the system ($S_M$) has a complex Lie algebra of dimension 10 generated by the holomorphic vector fields:
\begin{eqnarray}\label{273}
&&\partial x, \quad \partial y,\quad \partial u, \quad -2u\partial x+x\partial y,\quad x\partial u-2y\partial x,\\ \nonumber 
&&\quad \dfrac{x \partial x}{2}+y\partial y, \quad u\partial u+\dfrac{x \partial x}{2},\quad -2ux\partial x+\dfrac{1}{2}x^2\partial y-2 u ^2\partial u,\\ \nonumber 
&&\quad \dfrac{1}{2}xy\partial x+\dfrac{1}{2}y^2\partial y-\dfrac{1}{8} x ^2\partial u, \quad (\dfrac{x^2}{2}-2yu) \partial x+xy\partial y+xu\partial u.
\end{eqnarray}
\section{Conclusions}
In this paper we show that any infinitesimal Lie symmetry of a third PDE sysytem that is coming from CR-geometry is uniquely determined by ten initial Taylor coefficients. Also, by giving an example we show that this estimate is precise.

\end{document}